\documentclass[12pts]{amsart}
\usepackage{graphicx,fancybox,amssymb,color}

\newtheorem{theorem}{Theorem}[section]
\newtheorem{corollary}[theorem]{Corollary}

\newtheorem{lemma}[theorem]{Lemma}
\newtheorem{proposition}[theorem]{Proposition}

\theoremstyle{definition}

\theoremstyle{remark}
\newtheorem{remark}[theorem]{Remark}
\theoremstyle{remark}

\numberwithin{equation}{section}
\newcommand{\eps}{\varepsilon}

\usepackage{dsfont}
\newcommand{\RR}{\mathds{R}}

\newcommand{\cc}{\mbox{\fontfamily{phv}\selectfont C}}

\newcommand{\pp}{\mbox{\fontfamily{phv}\selectfont P}}

\newcommand{\hh}{\mbox{\fontfamily{phv}\selectfont H}}
\newcommand{\gen}{\mathbb{A}}
\newcommand{\genP}{\widetilde{\mathbb{A}}}

\date{Printed \today. (\jobname.tex)}

\begin{document}

\title{Infinitesimal generators of $q$-Meixner processes}

\date{%
Printed \today. File \jobname.tex}

\author{
W{\l}odek  Bryc
}

\address{
Department of Mathematics,
University of Cincinnati,
PO Box 210025,
Cincinnati, OH 45221--0025, USA}
\email{Wlodzimierz.Bryc@UC.edu}

\author{Jacek Weso{\l}owski}
\address{ Faculty of Mathematics and Information Science\\
Warsaw University of Technology\\ pl. Politechniki 1\\ 00-661
Warszawa, Poland}
\email{wesolo@alpha.mini.pw.edu.pl}


\begin{center}
\end{center}
\begin{abstract}
 We  show that the weak infinitesimal generator of a class of Markov processes acts on bounded continuous functions with bounded continuous second derivative as a singular integral with  respect to the orthogonality  measure of the explicit family of polynomials.
\end{abstract}
\maketitle

\section{Introduction}
We are interested in the class of non-homogeneous Markov processes that were introduced in \cite{Bryc-Wesolowski-03} under the name $q$-Meixner processes.  The  transition probabilities $\{P_{s,t}(x,dy):\,s<t, \,x\in\RR\}$ of a $q$-Meixner process with parameters $\tau\geq 0$,  $\theta\in\RR$, and $q\in[-1,1]$ are defined as the unique orthogonality measures of the polynomials $Q_n(y|x,s,t)$ in variable $y$ which solve the three step recurrence

\begin{multline}\label{Q-rec-G}
\displaystyle y Q_n(y|x,s,t) = Q_{n+1}(y|x,s,t)   + (\theta[n]_q+xq^n)
Q_n(y|x,s,t)
\\ +(t-sq^{n-1}+\tau[n-1]_q)[n]_q Q_{n-1}(y|x,s,t),
\end{multline}
where $n\geq 1$, and $Q_{-1}(y|x,s,t)=0$, $Q_0(y|x,s,t)=1$, so $Q_1(y|x,s,t)=y-x$.

(Recall the  $q$-notation: $[n]_q=\sum_{j=0}^{n-1}q^j$. The Chapman-Kolmogorov equations hold by  \cite[Proposition 3.2]{Bryc-Wesolowski-03}.)

For $-1<q<1$ recurrence  \eqref{Q-rec-G} can be reparametrized into a recurrence for the so called Al-Salam---Chihara polynomials, so the explicit formula for  $P_{s,t}(x,dy)$ can be read out from known results, see e.g. \cite{Ismail-05,Koekoek-Swarttouw}. However, the explicit form of the transition probabilities does not play a role in our proofs, and the expressions for the transition probabilities  are rather complicated, as they may have both the discrete and the absolutely continuous parts.

Cases $q=-1$ and $q=1$ are included in \eqref{Q-rec-G}. In the first case the recursion degenerates with $P_{s,t}(x,dy)$ supported on two points. In the second case polynomials $\{Q_n(y|x,s,t):n=0,1,\dots\}$   are the  reparametrization of the Meixner polynomials, and  we get L\'evy processes in the Meixner class \cite{schoutens2000stochastic}. {Since the infinitesimal generators of L\'evy processes are well understood, in this paper we concentrate on the case $-1< q<1$, see Remark \ref{R1.2}.}

{
  Several other special cases have appeared in the literature and have been studied by other authors.
 If $q=0$, then
   the      corresponding $q$-Meixner  Markov processes  arise as the so called classical versions of  the free-Meixner L\'evy processes; that is, the time ordered moments of Borel functions coincide, see  \cite[Definition 4.1]{BKS97}).
   For more details, including connections with  \cite{Biane98}  see \cite[Appendix, Note 2 and Note 3 ]{Bryc-free-gen:2009}.
    If $\theta=\tau=0$ then %
           the corresponding $q$-Meixner  Markov processes  arises as the classical version of the
  noncommutative $q$-Brownian motion; this can be seen by comparing the transition probabilities in \cite[Theorem 4.6]{BKS97} and in \cite[Section 4.1]{Bryc-Wesolowski-03}.
 }

{
Finally we note that from \cite[Proposition 3.3]{Bryc-Wesolowski-03} it follows that $q$-Meixner processes are  examples of (nonhomogeneous)   "polynomial processes"  studied in \cite{cuchiero2012polynomial,Szablowski-2012} with explicit  "time-space harmonic polynomials"  \cite{sengupta2000time}.
}

\subsection{Infinitesimal generators}
Inhomogeneous  Markov processes with state space $\RR$ are often turned into the homogeneous   Markov processes with state space $\RR\times[0,\infty)$ by considering $\widetilde{X}_t=(X_t,t)$, see for example \cite{Wentzell:1981fk}.  We will work in non-homogeneous setting as we use one-variable polynomials in some of the proofs.

From \eqref{Q-rec-G}  it is clear that   $Q_n(y|x,s,t)$ is a polynomial in $y,x,s,t$ with the leading term $y^n$ for every $n$. It follows that  the moments  $(x,s,t)\mapsto \int y^n P_{s,t}(x,dy)$ are polynomials in variables $x,s,t$ of degree at most $n$.

{We will be interested only in the case $-1< q<1$, in which case   probability measures $\{P_{s,t}(x,dy):s<t,x\in\RR\}$ are compactly supported, see Proposition \ref{Prop-properties}(v). Since for compactly supported measures, convergence of moments   implies weak convergence, the fact that conditional moments are polynomials in  variables $x,s,t$  implies that the  transition probabilities $\{P_{s,t}(x,dy):x\in\RR, 0\leq s<t \}$ define a Feller process. That is, if $f$ is  a bounded continuous function on $\RR$ then $x\mapsto  \int f(y)P_{s,t}(x,dy)$ is a bounded continuous function.}

We will denote by
$\pp_{s,t}$ the linear operators $f\mapsto \int f(y)P_{s,t}(x,dy)$. We will consider $\pp_{s,t}$ as a  contraction on various subspaces on Banach space $C_b(\RR)$ of bounded {continuous} functions with norm $\|f\|_\infty=\sup_{x\in\RR}|f(x)|$.
 We will  use the same symbol $\pp_{s,t}$ for the linear mappings on the vector space of all polynomials in variable $y$,  defined on monomials  by $y^n\mapsto \int z^n P_{s,t}(y,dz)$.

We will work with several notions of the infinitesimal generator.

The {\em weak left  infinitesimal generator} of a non-homogeneous Markov process with transition operators $\pp_{s,t}$  is defined for $t>0$ by
\begin{equation}
  \label{LLL-}
  \gen_t^- f=\lim_{h\to 0^+} \frac{1}{h}(\pp_{t-h,t} f-f).
\end{equation}
The domain  $\mathcal{D}_t^{-}$ of the weak left generator is the set of all $f\in C_b(\RR)$ where the convergence is pointwise and  the expression
$$\left\|\frac{1}{h}(\pp_{t-h,t} f-f)\right\|_\infty$$  under the limit   is bounded, compare \cite[Chapter 1, Section 6]{dynkin1965markov} for the homogeneous case.

 The {\em weak right infinitesimal generator} of a non-homogeneous Markov process with transition operators $\pp_{s,t}$  is defined for $t\geq 0$ by
the right-generator
\begin{equation}
  \label{LLL}
  \gen^+_t f=\lim_{h\to 0^+} \frac{1}{h}(\pp_{t,t+h} f-f).
\end{equation}
The domain  $\mathcal{D}_t^{+}$ of the weak right generator is the set of all $f\in C_b(\RR)$ where the convergence is pointwise and  the expression under the limit $\|\frac{1}{h}(\pp_{t,t+h} f-f)\|_\infty$ is bounded.

We will also consider \eqref{LLL-} and \eqref{LLL} with pointwise convergence on polynomials.  Since $\pp_{s,t}$ preserves the degree of a polynomial,  it is clear that the pointwise limits \eqref{LLL-}  and \eqref{LLL} exist for any polynomial $f$,  and that both limits are polynomials of degree at most $n$ in variable $x$. Thus when $f$ is a polynomial,  the limits \eqref{LLL-}  and \eqref{LLL}   define two linear operators  $ \genP_t^{\pm}$ that map polynomials to polynomials without increasing their degrees.

Our goal is to derive the common integral representation for these infinitesimal generators. We will write the generators as singular integrals with respect to an appropriate probability measure $\nu_{x,t}(dy)$
which we determine as an orthogonality measure of appropriate orthogonal polynomials.

For $x\in\RR$ and $t>0$, let $\nu_{x,t}(dy)$  be the orthogonality measure  of the following monic polynomials in real variable $y$. With
$W_{-1}(y;x,t)=0$, $W_{0}(y;x,t)=1$, for $n\geq 0$ consider polynomials
\begin{multline}
  \label{NuQ}
  y W_{n}(y;x,t)=W_{n+1}(y;x,t)+(\theta [n+1]_q + x q^{n+1})W_n(y;x,t)\\+((1-q)t+\tau)[n]_q[n+1]_q W_{n-1}(y;x,t).
\end{multline}
By Favard's theorem, these polynomials are orthogonal, and since for  $-1\leq q<1$ the coefficients of the recurrence are bounded, their orthogonality measure is compactly supported. In  fact, the three step recursion \eqref{NuQ} can be reparametrized into a recursion for the so called Al-Salam---Chihara polynomials \cite{Ismail-05,Koekoek-Swarttouw}.
(The dependence of measure $\nu_{x,t}(dy)$ on parameters $\theta,\tau,q$ is suppressed in our notation.)

  Our main result is  the following "singular integral" expression for the generator of the $q$-Meixner processes with $|q|<1$.

\begin{theorem}
  \label{Thm:gen_q_Meixner} Fix $\theta\in\RR$, $\tau\geq 0$ and $q\in(-1,1)$.
  \begin{enumerate}
    \item Let $f:\RR\to\RR$ be a bounded continuous function with bounded continuous second derivative. Then
  $f\in\mathcal{D}_t^{-}\cap \mathcal{D}_t^{+}$, both weak infinitesimal generators  coincide on $f$ and are given by
  \begin{equation}\label{gen-q}
  \gen_t^{\pm}(f)(x)=\frac{1}{2}f''(x)\nu_{t,x}(\{x\})+\int_{\RR\setminus\{x\}}\left( \frac{\partial}{\partial x}\frac{f(y)-f(x)}{y-x}\right)\nu_{x,t}(dy).
\end{equation}

\item If $f$ is a polynomial  then both left and right infinitesimal generators  ${\genP}_t^{\pm}$ coincide on $f$, and are given by the right hand side of \eqref{gen-q}. On polynomials, the latter expression takes a slightly simpler form
  \begin{equation}\label{gen-q+}
  \genP_t^{\pm}(f)(x)=\int_\RR \left(  \frac{\partial }{\partial x}\frac{f(y)-f(x)}{y-x}\right)\nu_{x,t}(dy).
  \end{equation}
  \end{enumerate}

\end{theorem}

\begin{remark} \label{R1.2} With minimal changes in our proofs, Theorem \ref{Thm:gen_q_Meixner} holds also for $q=1$. However, for $q=1$ the $q$-Meixner processes are L\'evy processes with finite moments, and the infinitesimal generators for centered L\'evy processes with finite second moments have been studied in more detail, see e.g. \cite{Anshelevich:2011}. In this case,  the restriction of the infinitesimal generator of a square-integrable Levy process  to our class of functions is given by   \eqref{gen-q} with   $ \nu_{x,t}(dy)=(\delta_{x}*K)(dy)$ where $K(dy)$ is the measure from the Kolmogorov representation for the characteristic function of the infinitely divisible measure $P_{0,1}(0,dy)$.
Kolmogorov measures $K(dy)$ for centered Levy processes in the Meixner class are known explicitly, see   \cite{schoutens1998levy}.  {They can also be read out  as the orthogonality measures of polynomials  \eqref{NuQ} for $x=0$,  $q=1$. }

\end{remark}
\begin{remark}\label{R1.3}
Theorem \ref{Thm:gen_q_Meixner} holds also for $q=-1$ with $
\nu_{x,t}(dy)=\delta_{\theta-x}$.
 Our proof could be modified to account for the possibility that  $[n]_q=0$ when $n$ is even. However, in this case the  transition probabilities  are supported on two points:
\begin{multline*}
P_{s,t}(x,dy)=\left(\frac12+\frac{\theta-2x}{2\sqrt{(\theta-2x)^2+4(t-s)}}\right)\delta_{\frac12(\theta-\sqrt{(\theta-2x)^2+4(t-s)})}\\+\left(\frac12-\frac{\theta-2x}{2\sqrt{(\theta-2x)^2+4(t-s)}}\right)\delta_{\frac12(\theta+\sqrt{(\theta-2x)^2+4(t-s)}) } \;.
\end{multline*}
So the fact that the weak infinitesimal generator on twice differentiable functions is given by  \eqref{gen-q} with the degenerate
$
\nu_{x,t}(dy)=\delta_{\theta-x}
$ is just an exercise, and we omit proof for this case.
\end{remark}

The following technical result is an intermediate step in the proof of Theorem \ref{Thm:gen_q_Meixner}.
\begin{theorem}
  \label{Lemma-Main}
  Measures
\begin{equation}\label{ProbMeas}
\frac{(y-x)^2}{t-s}P_{s,t}(x,dy)
\end{equation}
are probability measures and converge weakly as $s\to t^-$ to $\nu_{x,t}(dy)$. Similarly, probability measures \eqref{ProbMeas} converge weakly as $t\to s^+$ to $\nu_{x,s}(dy)$.
\end{theorem}
The proofs are in Section \ref{Sect:Thm(ii)}  and in Section \ref{Sect:proofII}. The plan of proof is as follows. We first prove Theorem \ref{Thm:gen_q_Meixner}(ii). We then  derive Theorem \ref{Lemma-Main} from Theorem \ref{Thm:gen_q_Meixner}(ii). Finally, we  show that Theorem \ref{Lemma-Main} implies Theorem \ref{Thm:gen_q_Meixner}(i).

We end this section with a short list of more explicit examples.
\subsection{Some special cases}\label{S1.2}  Measures $\nu_{x,t}(dy)$ take more explicit form in some special cases. In the corollaries, $\gen_t$ denotes either the  left or the  right weak infinitesimal
generator if it is applied to bounded continuous functions,  or one of the generators $\tilde{\gen}_t^\pm$ if it is acting on polynomials.

The generator of the $q$-Brownian process was determined \cite[Section 5, Theorem 23]{Anshelevich:2011}; his result  inspired our study of generators for more general $q$-Meixner processes.
\begin{corollary} \label{ThAnsh}
The infinitesimal generator of the $q$-Wiener process acts  on a polynomial $f$  or on a bounded continuous function $f$ with bounded continuous second derivative  as follows:
  \begin{equation}
    \label{gen-qBrown}
    (\gen_t f)(x)= \int \left(\frac{\partial  }{\partial x }\frac{f(  y)-f(x)}{ y -x}\right) P_{q^2t,t}(qx,dy).
  \end{equation}
Here $P_{s,t}(x,dy)$ denotes the transition probability measure of the $q$-Wiener process.
\end{corollary}
\begin{proof}
From \eqref{Q-rec-G} with  $\theta=\tau=0$, we read out that the transition probabilities $P_{s,t}(x,dy)$ of the $q$-Brownian motion are the orthogonality measures of the monic polynomials $\{Q_n(y|x,s,t)n\geq 0\}$ in variable $y$ which are given by the three step recurrence
\begin{multline}\label{Q-rec-G-q-W}
  yQ_n(y|x,s,t) = Q_{n+1}(y|x,s,t) + x q^nQ_n(y|x,s,t)\\+(t-sq^{n-1})[n]_q Q_{n-1}(y|x,s,t), \;  n\geq 1,
\end{multline}
with $Q_{-1}(y|x,s,t)=0$, $Q_0(y|x,s,t)=1$. Comparing  \eqref{NuQ} with  $\theta=\tau=0$ and  \eqref{Q-rec-G-q-W} with  $s=q^2t$ and $x$ replaced by $qx$ we see that   $\nu_{x,t}(dy)=P_{q^2t,t}(qx,dy)$. So  \eqref{gen-qBrown} follows from \eqref{gen-q}.
\end{proof}

The free Brownian motion corresponds to  $q=0$ and has been studied in \cite[page 392]{Biane:1998b}. The domain of the closely related generator for the free Ornstein-Uhlenbeck process  is described in more detail in \cite[page 150]{BKS97}.
\begin{corollary}\label{Th-BS}
 For $q=0$, the infinitesimal generator of the $q$-Wiener process acts  on a polynomial $f$  or on a bounded continuous function $f$ with bounded continuous second derivative   as follows:
  \begin{equation}
    \label{free-brown-gen}
   ( \gen_t f)(x)= \int_{(-2,2)} \left(\frac{\partial}{\partial x}\frac{f(\sqrt{t} y)-f(x)}{\sqrt{t}y -x}\right)\sqrt{4-y^2}dy/\pi.
  \end{equation}
\end{corollary}

\begin{proof}
  With $q=0$, this follows formula \eqref{gen-qBrown}: $P_{0,t}(0,dy)$ is the univariate law of the free Brownian motion $X_t$  started at $X_0=0$, which is know to be the semicircle law of mean 0 and variance $t$. Then $X_t/\sqrt{t}$ has the semicircle law of variance 1, so by a change of variable
  $$\int f(y) P_{0,t}(0,dy)= \int f(\sqrt{t}y)\sqrt{4-y^2}dy/\pi$$ for any (say polynomial) $f$.

\end{proof}

The $q$-Meixner processes with $q=0$ arise as classical versions of certain free L\'evy processes. The
 generator of such  processes was studied in \cite{Bryc-free-gen:2009}. Anshelevich \cite[Theorem 15]{Anshelevich:2011} determined the strong infinitesimal  generators for the more general class of Markov processes that arise from arbitrary  free L\'evy processes, and identified a large subset of their domain.

\begin{corollary}%
The generator
of $q$-Meixner process for $q = 0$
acts on a polynomial $f$
 or on  a bounded continuous function $f$ with bounded continuous second derivative as follows:
\begin{equation}\label{generator}
(\gen_t f)(x)=\int_{(\theta-2\sqrt{t+\tau},\theta+2\sqrt{t+\tau})}\left( \frac{\partial}{\partial x} \frac{f(y)-f(x)}{y-x} \right)w_{\theta,t+\tau}(dy),
\end{equation}
where $w_{m,\sigma^2}(y)\sim \sqrt{4 \sigma^2-(y-m)^2}$ is the semicircle density of mean $m$ and variance $\sigma$.
\end{corollary}

\begin{proof}
For $q=0$, recurrence  \eqref{NuQ} becomes
\begin{eqnarray*}
 W_{1}(y;x;t)&=&y-\theta ,\\
 (y-\theta) W_{n}(y;x,t)&=&W_{n+1}(y;x;t)+ ( t+\tau)  W_{n-1}(y;x,t), \mbox{ $n\geq 1$}.
\end{eqnarray*}

 The corresponding probability measure $\nu$ is the semicircle law of mean $\theta$ and variance $t+\tau$.

\end{proof}

\subsection{Some additional  observations}
Here we collect some simple "regularity" properties of the transition operators for the $q$-Meixner Markov processes.
\begin{proposition}\label{Prop-properties}
Suppose $(X_t)$ is a $q$-Meixner Markov process with transition probabilities $P_{s,t}(x,dy)$ that are orthogonality measures of polynomials \eqref{Q-rec-G}. Then the following properties hold.
\begin{enumerate}
\item Process $(X_t)$ is a martingale: $$\int y P_{s,t}(x,dy)=x.$$
\item More generally, $M_n(y;t):=Q_n(y|0,0,t)$ are martingale polynomials: $$\int M_n(y;t)P_{s,t}(x,dy)=M_n(x;s).$$
\item For $s<t$, and fixed $x\in\RR$, the positive measure  \eqref{ProbMeas}
 is a probability measure.
\item For  $s<t$,  fixed $x\in\RR$ and $U=(x-\delta,x+\delta)$ with $\delta>0$, we have
$$P_{s,t}(x,U)\geq 1-(t-s)/\delta^2.$$
\item For fixed $s<t$, $x\in\RR$ and $-1\leq q<1$, probability measure  $P_{s,t}(x,dy)$ has compact support
{
\item Transition probabilities  $P_{s,t}(x,dy)$ have Feller property:  if $f:\RR\to\RR$ is a bounded continuous function and $s<t$ then $g(x):= \int f(y)P_{s,t}(x,dy)$ defines a bounded continuous function $g:\RR\to\RR$. Furthermore, if $\lim_{x\to\pm \infty}f(x)=0$ then  $\lim_{x\to\pm \infty}g(x)=0$
}
\end{enumerate}
\end{proposition}
\begin{proof}
\begin{enumerate}
\item The martingale property follows from the fact that $Q_1(y|x,s,t)=y-x$ is orthogonal to $Q_0=1$ with respect to the measure $P_{s,t}(x,dy)$.
\item The more general martingale property is \cite[Proposition 3.3]{Bryc-Wesolowski-03} (the polynomials $M_n(y;t)$ are not orthogonal unless $X_0=0$).

\item Clearly, this is a positive measure so we only need to verify that it integrates to $1$. Since $$Q_2(y|x,s,t)=(y-x)^2-(y-x) (\theta +(q-1) x) -(t-s)$$ is orthogonal to $Q_0=1$, we get constant  conditional variance  $$\int(y-x)^2P_{s,t}(x,dy)=t-s.$$  This shows that  \eqref{ProbMeas} is a probability measure.

\item By Chebyshev's inequality,  $$P_{s,t}(x,U')\leq \frac{1}{\delta^2}\int(y-x)^2 P_{s,t}(x,dy)=\frac{(t-s)}{\delta^2}.$$

\item Compact support follows from the fact that for $-1\leq q<1$ the coefficients of recurrence \eqref{Q-rec-G} are bounded, see e.g. \cite[Theorem 2.5.4]{Ismail-05}.

{
\item Clearly, for any bounded measurable $f$ we have $\|g\|_\infty\leq \|f\|_\infty$. As explained in the introduction, the fact that $g(x)$ is continuous follows from the fact that conditional moments  $\int y^n P_{s,t}(x,dy)$ are polynomials in variable $x$. To show that the transition operators preserve vanishing at infinity property, given $\eps>0$ choose $A>0$ such that  $\|f\|_\infty (t-s)/A^2<\eps$ and  $\sup_{y>A}|f(y)|<\eps$. Then
for $x>2A$ we have
$$
\int f(y)P_{s,t}(x,dy)=\int_{y>A} f(y)P_{s,t}(x,dy)+\int_{y\leq A}f(y)P_{s,t}(x,dy)
$$
and the first term is bounded by
$$\left|\int_{y>A} f(y)P_{s,t}(x,dy)\right|\leq  \sup_{y>A}|f(y)|<\eps.$$
Since $ \{y: y\le A\}\subset \{y:|y-x|\ge A\}$
for  $x>2A$,  the second term is bounded by Chebyshev's inequality that was already used in the proof of (iv):
\begin{multline*}
\left|\int_{y\leq A}f(y)P_{s,t}(x,dy)\right|\leq \| f\|_\infty P_{s,t}(x,\{y:  |y-x|\geq A\})
\\ \leq \|f\|_\infty (t-s)/A^2<\eps
.
\end{multline*}
The proof for the case $x\to -\infty$ is similar.
}
\end{enumerate}

\end{proof}

\section{Proof of Theorem \ref{Thm:gen_q_Meixner}(ii)} \label{Sect:Thm(ii)}
We begin by checking that both left and right infinitesimal generators $\genP_t^{\pm}$ coincide on the   polynomials. %
\begin{lemma}\label{Lem-M}
Suppose $M_n(y;t):=Q_n(y|0,0,t)$ are the martingale polynomials from Proposition \ref{Prop-properties}(ii). Then
\begin{equation}\label{DM}
{\genP}_t^{\pm}(M_n(\cdot;t))(x)=-\frac{\partial}{\partial t}M_n(x;t).
\end{equation}
In particular,  for any polynomial $p$ the left and right infinitesimal generators $\genP_t^\pm ({p})$ exist and are equal.   \end{lemma}
\begin{proof}
Since $M_n(y;t)$ is given by a special case of recurrence \eqref{Q-rec-G}, it is clear that $M_n(y;t)$ is a polynomial in $t$ and hence  it is differentiable with respect to $t$.

By the martingale property,  $$\int M_n(y,t)P_{s,t}(x,dy)-M_n(x;t)=M_n(x;s)-M_n(x;t)$$ for $s<t$, so
\begin{multline*}{\genP}_t^-(M_n(\cdot;t))(x)=\lim_{s\to t^-}\frac{1}{t-s}\left( \pp_{s,t}(M_n(\cdot;t))(x)-M_n(x;t)\right)\\=\lim_{s\to t^-}\frac{1}{t-s}\left( M_n(x;s)(x)-M_n(x;t)\right)=-\frac{\partial}{\partial t}M_n(x;t).
\end{multline*}

We now consider the right generator. Writing $$M_n(y;t)=\sum_{k=0}^n a_k(t)y^k$$ we have
\begin{multline*}\frac{1}{h}\left(\int M_n(y,t)P_{t,t+h}(x,dy)-M_n(x;t)\right)\\= \int \frac{M_n(y;t)-M_n(y;t+h)}{h}P_{t,t+h}(x,dy)\\
=\sum_{k=0}^n \frac{a_k(t)-a_k(t+h)}{h}\int y^k P_{t,t+h}(x,dy).
\end{multline*}
Since $$\int y^k P_{t,t+h}(x,dy)\to x^k \mbox{ as $h\to 0$ } $$and $$\frac{a_k(t)-a_k(t+h)}{h}\to -a_k'(t) \mbox{ as $h\to 0$},$$ the formula follows.

To prove the second part, we write $p(y)$ as a (finite) linear combination  $$a_0(t)M_0(y;t)+\dots+a_{n-1}(t)M_{n-1}(y;t) +a_n(t) M_{n}(y;t).$$ Then by linearity
$$\pp_{s,u}({p})(x)=\sum_{k=0}^n  a_k(t) \pp_{s,u}(M_k(\cdot;t))(x),$$  so $\tilde{\gen}^{\pm}_t({p})(x)=-\sum_{k=0}^n a_k(t) \frac{\partial}{\partial t}M_k(x;t)$.
\end{proof}
Since both left and right generators coincide on polynomials, from now on write $\genP_t ({p})$ for their common value.

Next, consider an auxiliary operator $\hh_t$ acting on a polynomial $p$ as the difference of $\genP_t$ applied to polynomial $yp(y)$ and $x \tilde{\gen}_t({p})(x)$. Informally,
\begin{equation}\label{Def-H}
\hh_t({p})(x):=\tilde \gen_t(yp(y))(x)-x\genP_t({p(y)})(x).
\end{equation}

\begin{lemma}\label{HofM}
\begin{equation}\label{HM}
\hh_t(M_k(\cdot;t))(x)=[k]_q M_{k-1}(x;t).
\end{equation}
\end{lemma}
\begin{proof}
From \eqref{Q-rec-G} we get the recurrence
$$
y M_k(y;t)=M_{k+1}(y;t)+\theta [k]_q M_{k}(y;t)+(t+\tau[k-1])[k]_qM_{k-1}(y;t).
$$
Therefore, by linearity and Lemma \ref{Lem-M} we have
\begin{multline*}
\tilde \gen_t (y M_k(y;t))(x)\\=-\frac{\partial}{\partial t}M_{k+1}(x;t)-\theta [k]_q \frac{\partial}{\partial t}M_{k}(x;t)-(t+\tau[k-1])[k]_q\frac{\partial}{\partial t}M_{k-1}(x;t).
\end{multline*}

On the other hand,
\begin{multline*}x  \genP_t ( M_k(y;t))(x)=- x\frac{\partial}{\partial t}M_{k}(x;t) = -\frac{\partial}{\partial t}(xM_{k}(x;t))\\=
-\frac{\partial}{\partial t}\left(M_{k+1}(x;t)+\theta [k]_q M_{k}(x;t)+(t+\tau[k-1])[k]_qM_{k-1}(x;t)\right).
\end{multline*}
 Subtracting the two expressions yields the answer.

\end{proof}

\begin{lemma}\label{Lemma_HfmNu}
With $\nu_{x,t}(dy)$ as in Theorem \ref{Thm:gen_q_Meixner}, we have
\begin{equation}\label{Nu2H}
\hh_t({p})(x)=\int \frac{p(y)-p(x)}{y-x}\nu_{x,t}(dy).
\end{equation}
\end{lemma}

\begin{proof} In view of \eqref{HM}, we only need to show that  for all $n=0,1,\dots$ we have
\begin{equation}\label{Nu_int_M}
\int \frac{M_n(y;t)-M_n(x;t)}{y-x}\nu_{x,t}(dy)=[n]_q M_{n-1}(x;t).
\end{equation}
We prove \eqref{Nu_int_M} by induction. Since $M_0=1$ and $M_1(y;t)=y$, it is clear that \eqref{Nu_int_M} holds for $n=0,1$.

Suppose now that \eqref{Nu_int_M} holds for  some $n\geq 1$ and for all the previous  integers.

The induction step relies on the following algebraic identities from \cite[Lemma 3.1]{Bryc-Wesolowski-03}
\begin{equation}\label{QQQ}
0=Q_{n}(x|x,t,t)=\sum_{k=0}^n\,\left[\begin{array}{c} n \\ k \end{array}\right]_q\,Q_{n-k}(0|x,t,0)M_k(x,t)
\end{equation} and
\begin{equation}\label{QM}
Q_{n+1}(y|x,t,t)=\sum_{k=1}^{n+1}\,\left[\begin{array}{c} n+1 \\ k \end{array}\right]_q\,Q_{n+1-k}(0|x,t,0)(M_k(y;t)-M_k(x;t)).
\end{equation}
Here we use the $q$-notation $[n]_q!=\prod_{k=1}^n [k]_q$ with $[0]_q!=1$ and
$$\left[\begin{array}{c} n\\ k \end{array}\right]_q=\frac{[n]_q!}{[k]_q![n-k]_q!}.$$ (The latter is well defined as we only consider $q>-1$.)

From \eqref{Q-rec-G} applied to $s=t$ we see that for $n\geq 0$ we have $$Q_{n+1}(y|x,t,t)=(y-x)W_{n}(y;x,t),$$ where polynomials $\{W_n(y;x,t)\}$ satisfy recurrence \eqref{NuQ}.
 Identity \eqref{QM} gives
$$
W_{n}(y;x,t) =\sum_{k=1}^{n}\, \,\left[\begin{array}{c} n+1 \\ k \end{array}\right]_q\,Q_{n+1-k}(0|x,t,0)\,\tfrac{M_k(y;t)-M_k(x;t)}{y-x}.
$$

Since  $n\geq 1$, polynomial $W_n$ is orthogonal to $W_0=1$. Integrating the above equality with respect to the measure $\nu_{x,t}(dy)$
 we get
\begin{equation}\label{QMnu}
0=\sum_{k=1}^{n+1}\, \,\left[\begin{array}{c} n+1 \\ k \end{array}\right]_q\,Q_{n+1-k}(0|x,t,0)\,\int\,\tfrac{M_k(y,t)-M_k(x,t)}{y-x}\,\nu_{x,t}(dy).\end{equation}

We now observe that  identity \eqref{QQQ} gives
\begin{multline*}
\sum_{k=1}^{n+1}\, \,\left[\begin{array}{c} n+1 \\ k \end{array}\right]_q\,Q_{n+1-k}(0|x,t,0)[k]_q M_{k-1}(x,t)
\\=[n+1]_q\sum_{k=0}^n \,\left[\begin{array}{c} n \\ k \end{array}\right]_q\,Q_{n-k}(0|x,t,0) M_k(x,t)=0.
\end{multline*}
Subtracting the left hand side of the above  from \eqref{QMnu} we see that all but the last terms cancel by  the inductive assumption  \eqref{Nu_int_M}. The remaining term is
$$\left[\begin{array}{c} n+1 \\ n+1 \end{array}\right]_q Q_{0}(0|x,s,0)\left(\int\,\tfrac{M_{n+1}(y,t)-M_{n+1}(x,t)}{y-x}\,\nu_{x,t}(dy) - [n+1]_q M_{n}(x,t)\right)=0.$$
Since  $Q_{0}(0|x,s,0)=1$ and
$q\ne -1$, this proves that \eqref{Nu_int_M} holds for $n+1$, hence for all $n$, and \eqref{Nu2H} follows.
\end{proof}

\begin{lemma}\label{H2A}
If $\hh_t$, as defined on polynomials by \eqref{Def-H}, is given by an integral  \eqref{Nu2H}, then $\genP_t$ acts on polynomials by \eqref{gen-q+}.
\end{lemma}
\begin{proof}
 By linearity it is enough to prove  \eqref{gen-q+}  for $f(x)=x^n$. We proceed by induction. Since $\genP_t(1)=0$, the formula holds true for $n=0$.
Since
\begin{equation}\label{Diff}
\frac{\partial}{\partial x}\frac{f(y)-f(x)}{y-x}=\frac{f(y)-f(x)}{(y-x)^2}-\frac{ f'(x)}{y-x},
\end{equation}
assuming \eqref{gen-q+} holds for $x^n$, we have
\begin{multline}
 \tilde \gen_t(x^{n+1})=\hh_t(x^n)+x\genP_t(x^n) \\
 =\int \frac{y^n-x^n}{y-x} \nu_{x,t}(dy)+ \int  \left(x\frac{y^n-x^n}{(y-x)^2} - \frac{ n x^{n}}{y-x}\right)\nu_{x,t}(dy)
\\=\int \left( \frac{y^n(y-x)}{(y-x)^2} + x\frac{y^n-x^n}{(y-x)^2} - \frac{ (n+1) x^{n}}{y-x}\right)\nu_{x,t}(dy)\\
= \int \left( \frac{y^{n+1}-x^{n+1}}{(y-x)^2} - \frac{ (n+1) x^{n}}{y-x}\right)\nu_{x,t}(dy)\\
= \int \left(\frac{\partial}{\partial x} \frac{y^{n+1}-x^{n+1}}{y-x}  \right)\nu_{x,t}(dy).
\end{multline}
\end{proof}
\begin{proof}[Proof of Theorem \ref{Thm:gen_q_Meixner}(ii)]
By Lemma \ref{Lem-M} both left and right generators coincide on polynomials.
 The integral representation  follows by combining Lemma \ref{Lemma_HfmNu} with Lemma \ref{H2A}.
\end{proof}

\section{Proof of Theorem \ref{Lemma-Main} and  Theorem \ref{Thm:gen_q_Meixner}(i)}\label{Sect:proofII}

We will deduce Theorem \ref{Lemma-Main} from Theorem \ref{Thm:gen_q_Meixner}(ii), which we have already proved.  Consider operator $\hh_t$   defined in \eqref{Def-H}, and let $\cc_t$ be defined by the similar expression: $$\cc_t({p})(x)=\hh_t(y p(y))(x)-x\hh_t(p(y))(x).$$
\begin{lemma}\label{Lem-C}
If $\hh_t$ is given by  \eqref{Nu2H}, then $\cc_t({p})(x)=\int p(y)\nu_{x,t}(dy)$.
\end{lemma}
\begin{proof} Recall that if $p$ is a polynomials then $(p(y)-p(x))/(y-x)$ is a polynomial in variable $y$.
By simple algebraic manipulations we get
\begin{multline*}\hh_t(y p(y))(x)-x\hh_t(p(y))(x)=\\
\int \frac{yp(y)-xp(x)}{y-x}\nu_{x,t}(dy)-x \frac{p(y)-p(x)}{y-x}\nu_{x,t}(dy)=\int \frac{(y-x)p(y)}{y-x}\nu_{x,t}(dy).
\end{multline*}
\end{proof}

\begin{proof}[Proof of Theorem \ref{Lemma-Main}]
We give the proof for the first part only.  By Theorem \ref{Thm:gen_q_Meixner}(ii),  for a polynomial $f$ we have
\begin{multline*}
\cc_t(f)(x)=\hh_t(y f(y))(x)-x\hh_t(f(y))(x)=\\  \genP_t(y^2f(y))(x)-2x  \genP_t (yf(y))(x)+x^2  \genP f(y)(x)\\=\lim_{h\to 0^+}\int \frac{y^2 f(y)-x^2f(x)}{h}P_{t,t+h}(x,dy)-2x
\lim_{h\to 0^+}\int \frac{y f(y)-xf(x)}{h}P_{t,t+h}(x,dy)\\+x^2\lim_{h\to 0^+}\int \frac{ f(y)-f(x)}{h}P_{t,t+h}(x,dy)=
\lim_{h\to 0^+}\int \frac{(y-x)^2f(y)}{h}P_{t,t+h}(x,dy).
\end{multline*}

In view of Lemma \ref{Lem-C}, this shows that as $h\to 0$, all moments of probability measure $$\frac{(y-x)^2}{h}P_{t,t+h}(x,dy)$$ converge to the moments of measure $\nu_{x,t}(dy)$.  Probability measure $\nu_{x,t}(dy)$ is compactly supported (recall that $|q|<1$), so  it is uniquely determined by moments, and weak convergence follows.

\end{proof}
\begin{proof}[Proof of Theorem \ref{Thm:gen_q_Meixner}(i)]
If $f$ is bounded and has bounded and continuous second derivative then for a fixed $x$
$$
\varphi(y)=\begin{cases}
\frac{\partial}{\partial x}\frac{f(y)-f(x)}{y-x} & \mbox{ if $y\ne x$}\\
\frac{1}{2}f''(x) & \mbox{if $y=x$}
\end{cases}$$
is a bounded continuous function.
Indeed, by Taylor's theorem
$$\left|\frac{\partial}{\partial x}\frac{f(y)-f(x)}{y-x}\right|=\frac{1}{(y-x)^2}\left|\int_x^y f''(z)(z-x)dz\right|\leq \frac12\sup_{z\in\RR}|f''(z)|.$$

Next, we observe  that since $\int y P_{s,t}(x,dy)=x$, from \eqref{Diff}  we get %
\begin{multline*}
\frac{1}{t-s}\int \left(f(y)-f(x)\right)P_{s,t}(x,dy)=\frac{1}{t-s}\int_{\RR-\{x\}} \left(f(y)-f(x)\right)P_{s,t}(x,dy) \\=\int_{\RR-\{x\}}\left( \frac{\partial}{\partial x}\frac{f(y)-f(x)}{y-x}\right)\frac{(y-x)^2}{t-s}P_{s,t}(x,dy)=\int_\RR  \varphi(y)\frac{(y-x)^2}{t-s}P_{s,t}(x,dy).
\end{multline*}

Therefore, by Lemma \ref{Lemma-Main},
\begin{multline*}
\lim_{s\to t^-} \frac{1}{t-s}\int_\RR \left(f(y)-f(x)\right)P_{s,t}(x,dy)=\lim_{s\to t^-}  \int_\RR \varphi(y)\frac{(y-x)^2}{t-s}P_{s,t}(x,dy)\\=\int_\RR \varphi(y) \nu_{x,t} (dy)=
\frac{1}{2}f''(x)\nu_{t,x}(\{x\})+\int_{\RR\setminus\{x\}}\left( \frac{\partial}{\partial x}\frac{f(y)-f(x)}{y-x}\right)\nu_{x,t}(dy).
\end{multline*}
Note that by Proposition \ref{Prop-properties}(iii),
$$
\sup_{0\leq s<t}\sup_{x\in\RR}\left| \int_\RR  \varphi(y)\frac{(y-x)^2}{t-s}P_{s,t}(x,dy)\right|\leq \|\varphi\|_\infty,
$$
so $f$ is indeed in the domain of the weak generator.
This proves \eqref{gen-q} for the left generator. Similar argument applies to the right generator.
\end{proof}

 \subsection*{Acknowledgement}
We would like to thank W. M\l otkowski  for a helpful discussion that lead to   \eqref{NuQ}. {Comments by the referee helped us to improve the presentation.}
WB research was partially supported by NSF grant \#DMS-0904720 and by the Taft Research Center. JW research was partially supported by
NCN grant 2012/05/B/ST1/00554.

\def\cprime{$'$}

\end{document}